\documentclass[11pt]{amsart}   
\usepackage{graphicx, amsmath, amssymb, amsthm, amscd}
\usepackage{epsf}
\usepackage{color}
\DeclareGraphicsExtensions{.jpg,.pdf,.mps,.png}
\newtheorem{theorem}{Theorem}[section]

\newtheorem{claim}[theorem]{Claim}
\newtheorem{corollary}[theorem]{Corollary}
\theoremstyle{definition}

\theoremstyle{remark}

\numberwithin{equation}{section}

\newtheorem*{question*}{Question}
\theoremstyle{example}
\newtheorem{example}[theorem]{Example}

\begin{document}

\vspace{0.5in}

\title[On a generalization of the Cartwright-Littlewood]%
{On a generalization of the Cartwright-Littlewood fixed point theorem for planar homeomorphisms}

\author{J. P. Boro\'nski}
\address[J. P. Boro\'nski]{National Supercomputing Centre IT4Innovations, Division of the University of Ostrava,
Institute for Research and Applications of Fuzzy Modeling,
30. dubna 22, 70103 Ostrava,
Czech Republic -- and -- Faculty of Applied Mathematics,
AGH University of Science and Technology,
al. Mickiewicza 30,
30-059 Krak\'ow,
Poland}
\email{jan.boronski@osu.cz}





\begin{abstract} 
 We prove a generalization of the fixed point theorem of Cartwright and Littlewood. Namely, suppose $h : \mathbb{R}^2 \to\mathbb{R}^2$ is an orientation preserving planar homeomorphism, and let $C$ be a continuum such that $h^{-1}(C)\cup C$ is acyclic. If there is a $c\in C$ such that $\{h^{-i}(c):i\in\mathbb{N}\}\subseteq C$, or $\{h^i(c):i\in\mathbb{N}\}\subseteq C$, then $C$ also contains a fixed point of $h$. Our approach is based on Morton Brown's short proof of the result of Cartwright and Littlewood. In addition, making use of a linked periodic orbits theorem of Bonino we also prove a counterpart of the aforementioned result for orientation reversing homeomorphisms, that guarantees a $2$-periodic orbit in $C$ if it contains a $k$-periodic orbit ($k>1$).
\end{abstract}

\maketitle
\noindent
\section{Introduction}
In  1951 M.L. Cartwright and J.E Littlewood \cite{CaLi} studied van der Pol's differential equation and were led to investigate the existence of fixed points of planar homeomorphisms in invariant continua. The continuum that they encountered had boundary that was not locally connected and was potentially even indecomposable. Recall that a {\itshape continuum} is a connected and compact nondegenerate set. It is {\itshape indecomposable} if it is not the union of any two proper subcontinua. A planar continuum is {\itshape acyclic} if it does not separate the plane. The following is the celebrated Cartwright-Littlewood fixed point theorem.
\vspace{0.25cm}

\noindent
{\bfseries Theorem A.} {\itshape (Cartwright\&Littlewood \cite{CaLi}) Let $f : \mathbb{R}^2 \to  \mathbb{R}^2$ be an orientation preserving planar homeomorphism. Suppose there is an acyclic continuum $C$ invariant under $f$; i.e. $f(C)=C$. Then there is a fixed point of $f$ in $C$.}
\vspace{0.25cm}

A number of authors provided alternative proofs of the result, many of which are substantially shorter then the original one, including those given by M. Brown \cite{BM} and O.H. Hamilton \cite{Ha}. M. Barge and R. Gillette proved that the continuum considered by Cartwright and Littlewood was in fact indecomposable \cite{BG}. Barge and J. Martin also showed \cite{BaM2} that any inverse limit of arcs with a single continuous bonding map gives an acyclic continuum (called \textit{arc-like}\footnote{A continuum is called \textit{arc-like, chainable,} or \textit{snakelike} if it can be given as the inverse limit of arcs with continuous bonding maps.}) that is an attractor of a planar homeomorphism. The class of arc-like continua includes the famous hereditarily indecomposable \textit{pseudo-arc} \cite{Bi1} (that is homogeneous, contains no arcs and is nowhere locally connected) and Knaster buckethandle continuum \cite{AF} (cf. topological horseshoe \cite{Sh}). There is often a strong connection between topology of the attracting continua and complexity of the associated dynamics, like the link between indecomposabilty and chaos \cite{BM1} (a chaotic decomposable arc-like attractor of a planar homeomorphism without indecomposable subcontinua recently constructed by P. Oprocha and the author \cite{BO} is depicted in Figure \ref{fig:hd}). In addition, V. A. Pliss in \cite{Pl} showed that any acyclic plane continuum is the maximal bounded closed set invariant under a transformation $F$, where $F$ is a solution of certain dissipative system of differential equations. Note that Theorem A holds true for orientation reversing homeomorphism, by the result of H. Bell \cite{BH}, and K. Kuperberg generalized Bell's result to plane separating continua \cite{KK},\cite{KK2}. However, it remains a major 100-year-old open problem if every acyclic planar continuum has the fixed point property (see Scottish Book Problem 107 \cite{MD}, and \cite{BFMOT} for the most recent developments regarding the subject).  
\begin{figure}
	\centering
		\includegraphics[width=0.95\textwidth]{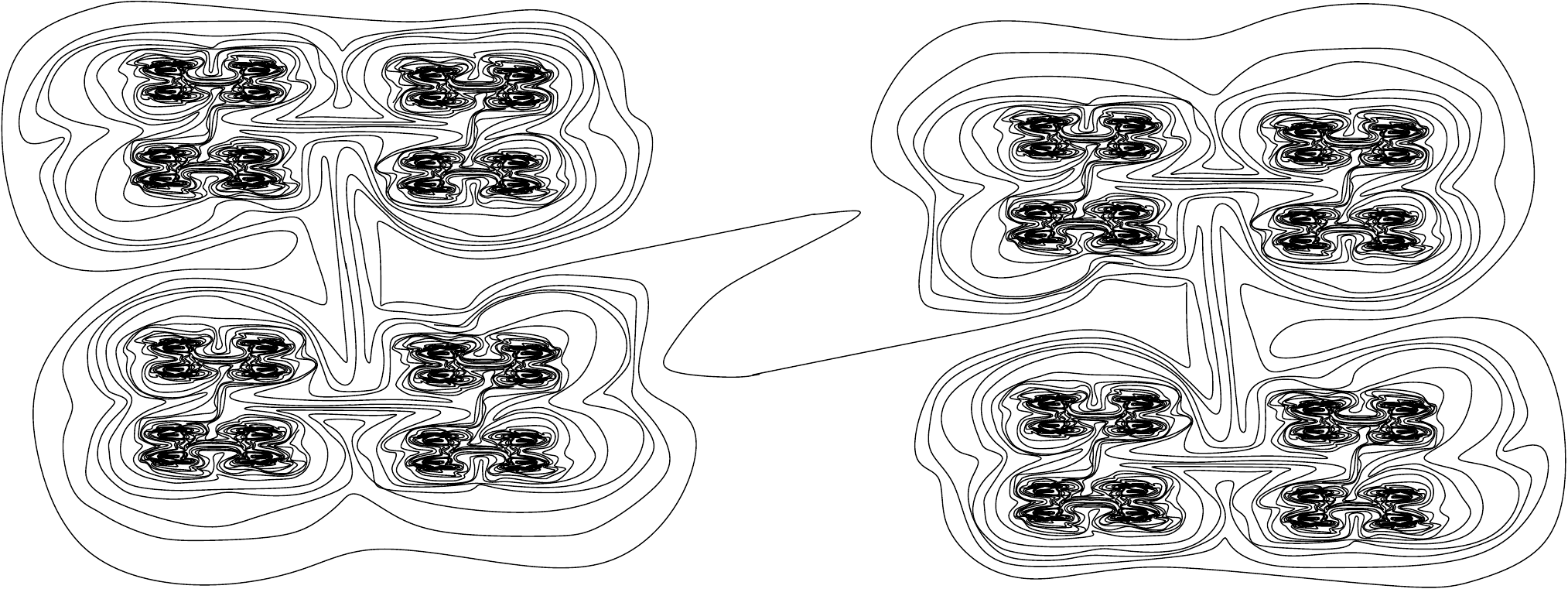}
	\label{fig:hd}
	\caption{An (acyclic) arc-like chaotic attractor without indecomposable subcontinua.}
\end{figure}

In the present paper we are interested in generalizing Theorem A in the following direction: suppose a continuum $C$ is not necessarily invariant under $h$; can one find some natural conditions under which a component of $C\cap h(C)$ contains a fixed point? Finding such conditions has a potential of giving a lower bound for the number of fixed points of $h$, that could prove very useful if $C\cap h(C)$ has more than just one component. Such conditions could also aid in locating fixed points of $h$. With this goal in mind, we shall demonstrate that given a continuum $C$ in order for $h$ to have a fixed point in $C$ it suffices that $h^{-1}(C)\cup C$ is acyclic and there is a backward or forward orbit $\mathcal{O}$ entirely contained in $C$. As a consequence we get that if $h$ has no fixed point in $C$ then $C$ also does not contain any periodic orbits. A related result has been recently obtained by G. Ostrovski \cite{Os}, with very strong assumptions on the topological structure of $C$. Namely, by entirely different methods, Ostrovski proved the following two theorems. 
\vspace{0.25cm}

\noindent
{\bfseries Theorem B.} {\itshape (Ostrovski \cite{Os}) Let $X\subset \mathbb{R}^2$ be a compact, simply connected, locally connected subset of the real plane and let $f : X \to Y \subset \mathbb{R}^2$ be a
homeomorphism isotopic to the identity on $X$. Let $C$ be a connected component of $X \cap Y$. If $f$ has a periodic orbit in $C$, then f also has a fixed point in $C$.}
\vspace{0.25cm}

\noindent
{\bfseries Theorem C.} {\itshape (Ostrovski \cite{Os}) Let $D\subset\mathbb{R}^2$ be a Jordan domain and $f: D \to E\subset \mathbb{R}^2$ an orientation preserving homeomorphism. Let $C$ be a connected
component of $D \cap E$. If $f$ has a periodic orbit in $C$, then $f$ also has a fixed point in $C$.}

Theorem C is in fact a corollary to Theorem B, since it is known that any homeomorphism $f:X\to X$ satysfying the assumptions of Theorem C is isotopic to the identity on $D$ (see \cite{Os} for details). However, Theorem C better encapsulates the motivation for our results, because we formulate them for a rather wide class of continua, which may contain no arcs at all, and therefore they may admit no isotopies connecting two distinct homeomorphisms. Our first result is the following theorem. 
\begin{theorem}\label{main1}
Let $h : \mathbb{R}^2 \to  \mathbb{R}^2$ be an orientation preserving planar homeomorphism, and let $C$ be a continuum such that $h^{-1}(C)\cup C$ is acyclic. If there is a point $c\in C$ such that $\{h^{-i}(c):i\in\mathbb{N}\}\subseteq C$, or
$\{h(c):i\in\mathbb{N}\}\subseteq C$, then $C$ also contains a fixed point of $h$. 
\end{theorem}
Note that Theorem \ref{main1} generalizes Theorem A because if $h(C)=C$ then $C\cup h^{-1}(C)=C$ and $\{h^{-i}(c):i\in\mathbb{Z}\}\subseteq C$ for any $c\in C$. It is also quite clear that Theorem \ref{main1} implies Theorem B in case $X$ is 1-dimensional, because in this case $h^{-1}(C)\cup C$ is acyclic as a subcontinuum of the 1-dimensional and simply connected $X$. Finally, Theorem \ref{main1} is false without the assumption that $h^{-1}(C)\cup C$ is acyclic. Indeed, it is enough to consider the rotation $h$ of the plane about the origin by $180^o$, and let $C$ to be the upper semicircle in the unit circle. Then $C$ contains a 2-periodic orbit (the end points) but no fixed point. 

At this point we would like to point out that some different fixed point results for non-invariant continua were proven in \cite[Chapter 5]{BFMOT} by Blokh, Fokkink, Mayer, Oversteegen, and Tymchatyn. The authors showed that every positively oriented map of the complex plane that strongly scrambles the boundary of an acyclic (potentially non-invariant) continuum $C$ must have a fixed point in $C$, and a related result is also proven for certain maps on dendrites. Since positively oriented maps generalize orientation preserving homeomorphisms their results are related to ours, but the emphasis of their approach is more on the structure of $X$, expressed in one-step "`geometric"' conditions, whereas our approach can be considered more of dynamical nature, as determined by iterations of the homeomorphism on $X$. The reader is referred to \cite{BFMOT} for more details.

A natural question to ask next is: what if $h$ reverses orientation? Can one also generalize Bell's Theorem \cite{BH} in the same way? This question seems a little bit more delicate, as we shall now explain.  First note that if $h$ reverses orientation then $h^2$ preserves orientation, and so it is natural to expect for $h$ either a fixed point or point of least period $2$, under assumptions similar to those of Theorem \ref{main1}. Note also that the fact that $h^{-1}(C)\cup C$ is acyclic, for a continuum $C$, does not imply that $h^{-2}(C)\cup C$ is acyclic, so a straightforward application of Theorem \ref{main1} in the orientation reversing case does not automatically give us fixed points for $h^2$. Next recall that an important property of orientation preserving homeomorphisms, that is often used to prove the Cartwright-Littlewood theorem, is the result attributed to L. E. Brouwer \cite{Br}: if a planar homeomorphism has a bounded orbit then it also has a fixed point. This property is not shared by orientation reversing homeomorphisms. In 1981 S. Boyles \cite{By2} constructed an example of such a homeomorphism with EVERY orbit bounded but no fixed points. Inspired by her construction, in Section \ref{example} we will exhibit an example of an orientation reversing homeomorphism $H$ with every orbit bounded, and a continuum $C$, with $H^{-1}(C)\cup C$ acyclic,  that contains periodic points of any even period, but contains no fixed points (in fact $H$ is fixed point free). 

Nonetheless, one can also hope for a counterpart of Theorem \ref{main1} for orientation reversing homeomorphisms, if a periodic point of least period $2$ in $C$ is to be guaranteed instead of a fixed point. We obtain such a result in Theorem \ref{main2}, where we show that assuming a continuum $C$, with the property that $h^{-1}(C)\cup C$ acyclic, contains a periodic orbit of period $k>1$, one can infer the existence of a $2$-periodic point in $C$. This result does not seem to follow from Brown's work but can be proved using Bonino's powerful result on linked periodic orbits \cite{Bo2}. Very similar arguments were previously used by the author \cite{BoJ} to show that any orientation reversing planar homeomorphism with a $k$-periodic orbit ($k>1$) contained in an invariant acyclic continuum $C$ must have a 2-periodic orbit in $C$. 
\begin{theorem}\label{main2}
Let $g : \mathbb{R}^2 \to  \mathbb{R}^2$ be an orientation reversing planar homeomorphism, and let $C$ be a continuum which contains a periodic orbit of least period $k>1$. If $g^{-1}(C)\cup C$ is acyclic then $C$ contains a point of least period $2$.
\end{theorem}
Although, based on Theorem \ref{main1}, one could expect that there must be either a fixed point or a $2$-periodic point in $C$, one can see that Theorem \ref{main2} determines more than just a mere alternative. In particular, $C$ may or may not contain a fixed point, but it is guaranteed to contain a point of least period $2$. In the last section of our paper we show that the above result is in a sense the best possible. In particular, one can neither expect a fixed point under the assumptions, nor can one weaken the assumptions and require that a non-periodic orbit contained in $C$ will force a point of least period $2$. Note that, as a corollary to the above two theorems, we obtain the following.
\begin{corollary}\label{main3}
Let $h : \mathbb{R}^2 \to  \mathbb{R}^2$ be a planar homeomorphism and $X$ be a one-dimensional acyclic continuum. Suppose there are $n$ components of $X\cap h(X)$, each of which contains a periodic orbit of period $k>1$. Then:
\begin{enumerate}
\item if $h$ preserves orientation then $h$ has at least $n$ fixed points in $X$;
\item if $h$ reverses orientation then $h$ has at least $\left\lfloor \frac{n+1}{2}\right\rfloor$ orbits of period $2$ in $X$.
\end{enumerate}
\end{corollary}

\section{Preliminaries}
Given a set $D$ by $\partial D$ we shall denote the boundary of $D$. Given a homeomorphism $h:\mathbb{R}^2\to\mathbb{R}^2$ we define $h^1(x)=h(x)$ and $h^{k+1}(x)=h\circ h^k(x)$ for all $k\in\mathbb{N}$ and $x\in\mathbb{R}^2$. Similarly, $h^{-k-1}(x)=h^{-1}\circ h^{-k}(x)$ for every $k\in\mathbb{N}$. A point $x\in\mathbb{R}^2$ is said to be {\itshape $k$-periodic} or {\itshape of least period $k$} if $h^k(x)=x$ and $h^i(x)\neq x$ for $i=1,\ldots,k-1$. If $k=1$, that is $h(x)=x$, then we say that $x$ is a {\itshape fixed point} of $h$. The {\itshape forward orbit} of $x$ is given by $\{h^n(x):n=1,2,3,\ldots\}$ and {\itshape backward orbit} of $x$ is given by $\{h^{-n}(x):n=1,2,3,\ldots\}$. If $x$ is $k$-periodic then the orbit of $x$ is given by $\{x,h(x),\ldots,h^{k-1}(x)\}$ and is said to be a {\itshape $k$-periodic orbit}. Following \cite{Bo2}, we say that two periodic orbits $\mathcal{O}$ and $\mathcal{O}'$ are {\itshape linked} if one cannot find a Jordan curve $C\subseteq\mathbb{S}^2$ separating $\mathcal{O}$ and $\mathcal{O}'$ which is freely isotopic to $h(C)$ in $\mathbb{S}^2\setminus\left(\mathcal{O}\cup\mathcal{O}'\right)$. $C$ and $h(C)$ are {\itshape freely isotopic} in $\mathbb{S}^2\setminus\left(\mathcal{O}\cup\mathcal{O}'\right)$ if there is an isotopy $\{i_t:\mathbb{S}^1\rightarrow \mathbb{S}^2\setminus\left(\mathcal{O}\cup\mathcal{O}'\right):0\leq t\leq 1\}$ from $i_0(\mathbb{S}^1)=C$ to $i_1(\mathbb{S}^1)=h(C)$; i.e. $i_t(\mathbb{S}^1)$ is a Jordan curve for any $t$ ($\mathbb{S}^1$ denotes the unit circle). Also, recall that if $U$ is an open surface and $(\tilde{U},\tau)$ is its universal covering space then given a homeomorphism $h:U\to U$ there exists a {\itshape lift} homeomorphism $\tilde{h}:\tilde{U}\to \tilde{U}$ such that the following diagram commutes.
\[
\begin{CD}
\tilde{U}@>\tilde{h}>> \tilde{U} \\
@V\tau VV @V\tau VV \\
U @>h>>U
\end{CD}
\]
Additionally if $h(x)=y$ then $\tilde{h}$ is uniquely determined by the choice of two points $\tilde{x}\in\tau^{-1}(x)$,  $\tilde{y}\in\tau^{-1}(y)$ and setting $\tilde{h}(\tilde{x})=\tilde{y}$.
\section{Proofs of the main results}
Our proof of Theorem \ref{main1} will follow the approach used by Brown in \cite{BM} to prove the Cartwright-Littlewood fixed point theorem. His proof by contradiction was based on an idea that an orientation preserving homeomorphism $h:\mathbb{R}^2\to\mathbb{R}^2$ with an acyclic continuum $X=h(X)$ and the fixed point set $\operatorname{Fix}(h)$ can be lifted to a homeomorphism $\tilde{h}$ of the universal covering $(\tilde{U},\tau)$ of a component of $\mathbb{R}^2\setminus \operatorname{Fix}(h)$ that will have a bounded orbit, yet no fixed point. This leads to a contradiction with the Brouwer's theorem. The approach works because $\tilde{U}$ is homeomorphic to $\mathbb{R}^2$ and the lift $\tilde{h}$ preserves orientation. Strictly speaking the existence of a bounded orbit of $\tilde{h}$ is guaranteed by the fact that $X$ lifts to disjoint homeomorphic copies (the components of $\tau^{-1}(X)$) and $\tilde{h}$ can be chosen so that one such copy is invariant. Indeed, it is well known that every acyclic continuum $X$ is the intersection of a decreasing sequence of closed disks $\{D_n:n\in\mathbb{N}\}$; i.e. $D_n\subseteq D_{n-1}$ for every $n$, and $X=\bigcap_{n\in\mathbb{N}}D_n$. In particular, $X$ has arbitrarily small simply connected neighborhoods that can separate $X$ from $\operatorname{Fix}(h)$. 
\begin{figure}
	\centering
		\includegraphics[width=0.70\textwidth]{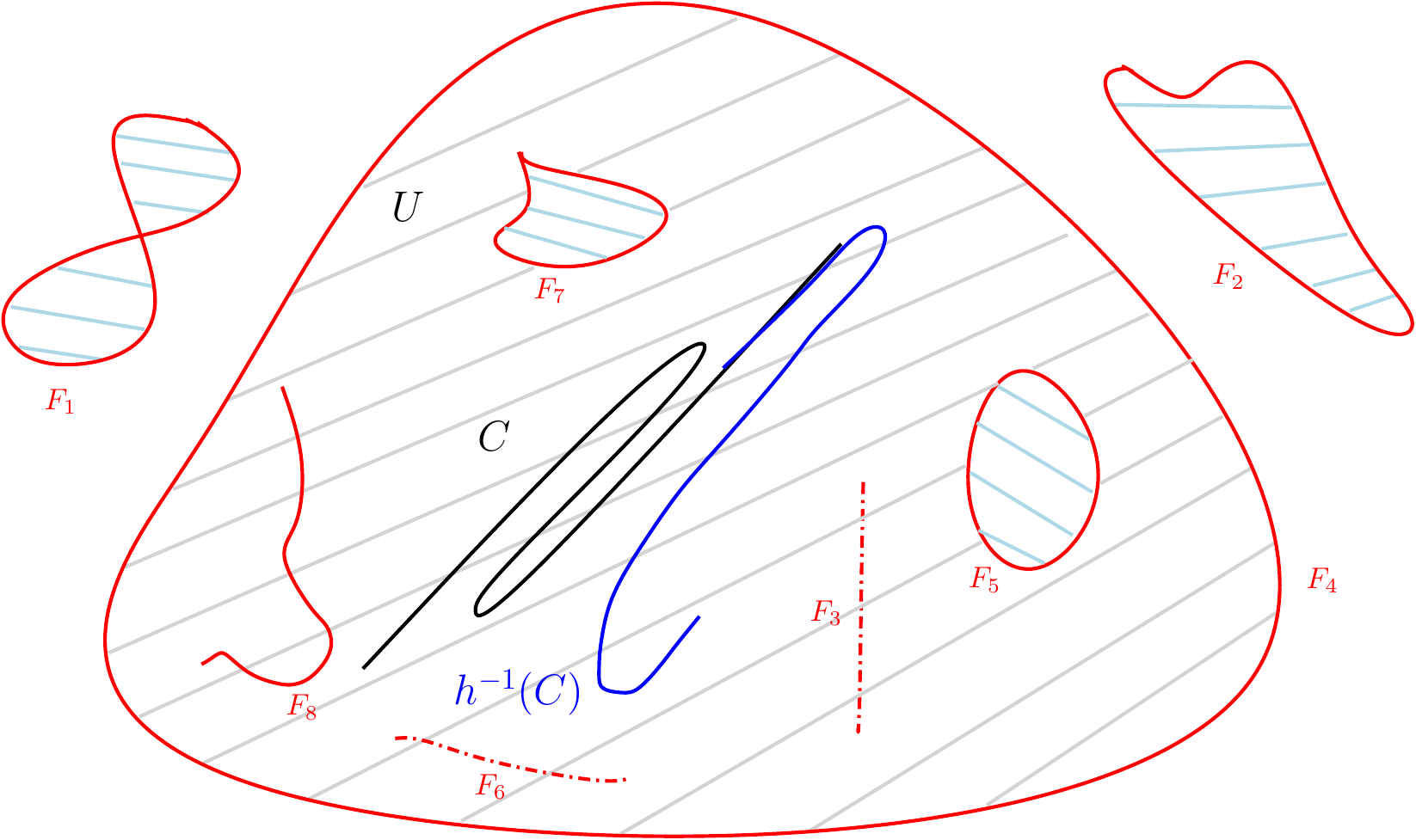}
	\label{fig:ostro2e}
	\caption{$U$ is the complementary domain of $\operatorname{Fix}(h)=\bigcup_{i=1}^8F_i$ that contains $C\cup h^{-1}(C)$.} 
\end{figure}

\begin{proof}(of Theorem \ref{main1})

Refer to Figure \ref{fig:ostro2e}. By contradiction, suppose that $\operatorname{Fix}(h)\cap C=\emptyset$. Let $U$ be the component of $\mathbb{R}^2\setminus \operatorname{Fix}(h)$ that contains $C$. Note that $h(U)=U$, since $h$ is onto, $h(U)$ must be contained in one of the complementary domains of $F$ and $h(U)\cap U\neq\emptyset$. Let $(\tau,\tilde{U})$ be the universal cover of $U$. Note that $\tilde{U}$ is homeomorphic to $\mathbb{R}^2$ and $C$ lifts to disjoint homeomorphic copies in $\tilde{U}$, being acyclic \cite{BM}. Let $\tilde{C}$ be one such a copy. 

CASE I. First suppose that $c$ has a backward orbit $\mathcal{O}=\{h^{-i}(c):i\in\mathbb{N}\}$ contained in $C$. Let $x=h^{-1}(c)$.  

There is a lift $\tilde{h}:\tilde{U}\to\tilde{U}$ of $h$ determined by $\tilde{h}(\tilde{C})\cap\tilde{C}\neq\emptyset$. Namely, let $\tilde{x}\in\tau^{-1}(x)\cap\tilde{C}$. If $\tilde{y}\in\tau^{-1}(c)\cap\tilde{C}$ then the condition $\tilde{h}(\tilde{x})=\tilde{y}$ uniquely determines the lift $\tilde{h}$. 
\begin{claim}\label{backwards}
$\tilde{h}^{-1}(\tilde{x})\in\tilde{C}$. 
\end{claim}
\begin{proof}(of Claim \ref{backwards})
Refer to Figure 3. We shall use the fact that under the covering map $\tau$ continua which map to acyclic continua must map one-to-one. Set $Q=h^{-1}(C)$ and let $\tilde{Q}$ be the component of $\tau^{-1}(Q)$ such that $\tilde{Q}=\tilde{h}^{-1}(\tilde{C})$. By contradiction suppose $\tilde{h}^{-1}(\tilde{x})\notin\tilde{C}$. Then there is another copy $\tilde{C}'$ of $C$ in $\tau^{-1}(C)$ such that $\tilde{h}^{-1}(\tilde{x})\in\tilde{C}'$. Note that $\tilde{h}^{-1}(\tilde{x})\in\tilde{Q}\cap\tilde{C}'$ and $\tilde{x}\in \tilde{Q}\cap\tilde{C}$. But then $\tilde{Q}\cup\tilde{C}\cup\tilde{C}'$ is a continuum with $\tau(\tilde{Q}\cup\tilde{C}\cup\tilde{C}')=C\cup h^{-1}(C)$, and $C\cup h^{-1}(C)$ is acyclic so $\tau^{-1}(C\cup h^{-1}(C))$ cannot contain two components of $\tau^{-1}(C)$. This leads to a contradiction.
\end{proof}
\begin{claim}\label{bounded}
$\tilde{h}^{-i}(\tilde{x})\in\tilde{C}$ for every $i$. 
\end{claim}
\begin{proof}(of Claim \ref{bounded})
This follows by induction. Suppose we have already proved that $\tilde{h}^{-i}(\tilde{x})\in\tilde{C}$ for all $i=1,\ldots,j$. Now replace $\tilde{x}$ with $\tilde{h}^{-j}(\tilde{x})$ in the proof of the above claim and repeat the same arguments. It follows that $\tilde{h}^{-j-1}(\tilde{x})\in\tilde{C}$ and then, by induction, that $\tilde{h}^{-i}(\tilde{x})\in\tilde{C}$ for every $i$.
\end{proof}
\begin{figure}
	\centering
		\includegraphics[width=0.90\textwidth]{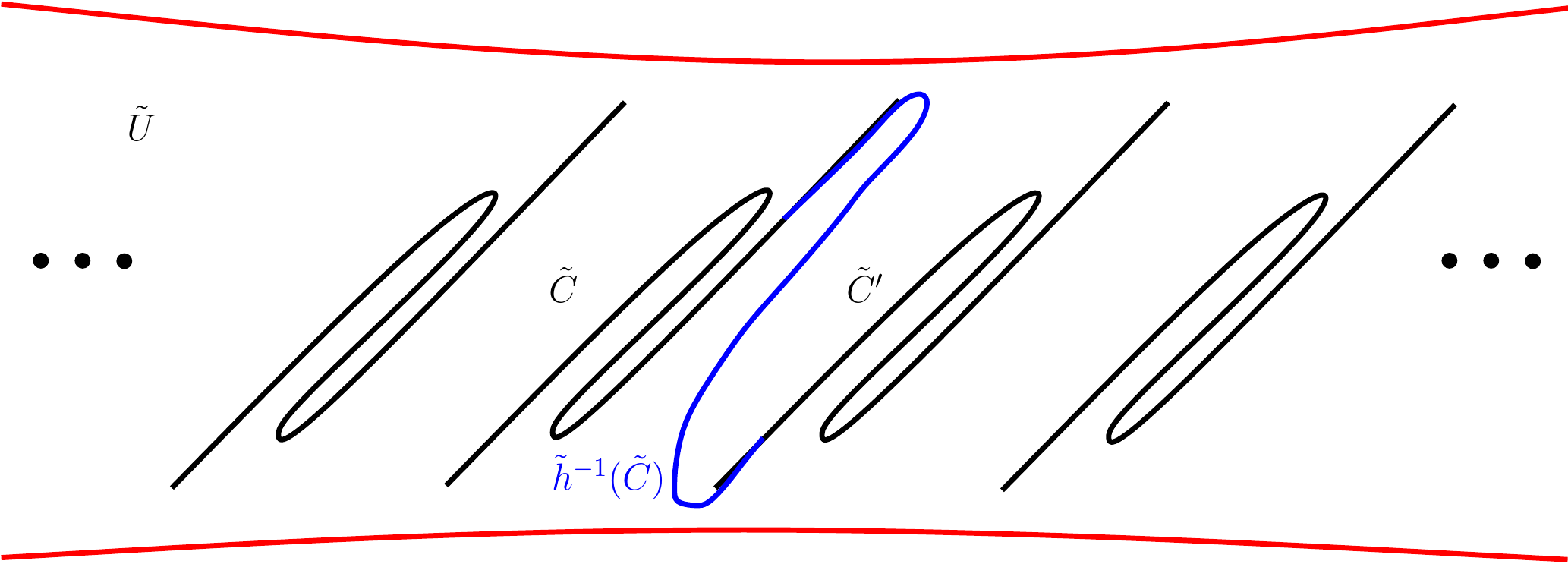}
	\label{fig:ostro2}
	\caption{Proof of Claim \ref{bounded}: $\tilde{h}^{-1}(\tilde{C})$ cannot intersect $\tilde{C}'$.} 
\end{figure}
To complete the proof of CASE I note that $\tilde{h}^{-1}:\tilde{U}\to\tilde{U}$ is a planar homeomorphism with $\{\tilde{h}^{-i}(\tilde{x}):i\in\mathbb{N}\}\subseteq \tilde{C}$. Since $\tilde{h}^{-1}$ has no fixed points we obtained a contradiction with Brouwer's Theorem \cite{Br}. 

CASE II. Now suppose that $c$ has a forward orbit $\mathcal{O'}=\{h^{i}(c):i\in\mathbb{N}\}$ contained in $C$. Then the proof of CASE I can be repeated because the fact that $h$ is a homeomorphism and $h^{-1}(C)\cup C$ is acyclic implies that $C\cup h(C)$ is acyclic as well.
\end{proof}
Now we shall prove Theorem \ref{main2}. Our proof, similar to \cite{BoJ} (proof of Theorem 3.1), will rely on the powerful result of Bonino on linked periodic orbits \cite{Bo2}. Recall that two periodic orbits $\mathcal{O}$ and $\mathcal{O}'$ of a homeomorphism $h$ are linked if one cannot find a Jordan curve $C\subseteq\mathbb{S}^2$ separating $\mathcal{O}$ and $\mathcal{O}'$ which is freely isotopic to $h(C)$ in $\mathbb{S}^2\setminus\left(\mathcal{O}\cup\mathcal{O}'\right)$. $C$ and $h(C)$ are freely isotopic in $\mathbb{S}^2\setminus\left(\mathcal{O}\cup\mathcal{O}'\right)$ if there is an isotopy $\{i_t:\mathbb{S}^1\rightarrow \mathbb{S}^2\setminus\left(\mathcal{O}\cup\mathcal{O}'\right):0\leq t\leq 1\}$ from $i_0(\mathbb{S}^1)=C$ to $i_1(\mathbb{S}^1)=h(C)$; i.e. $i_t(\mathbb{S}^1)$ is a Jordan curve for any $t$. 
\begin{proof}(of Theorem \ref{main2})
We shall show that, under the assumptions, if $C$ contains a $k$-periodic ($k>1$) orbit $\mathcal{O}$ then the 2-periodic orbit linked to $\mathcal{O}$, guaranteed by \cite{Bo2}, must intersect $C$. Note that if $k=2$ then there is nothing to prove, so assume $k>2$. Compactify $\mathbb{R}^2$ by a point $\infty$ to obtain $\mathbb{S}^2=\mathbb{R}^2\cup\{\infty\}$. Note that $g:\mathbb{R}^2\rightarrow\mathbb{R}^2$ can be extended to a homeomorphism $\tilde{g}:\mathbb{S}^2\rightarrow\mathbb{S}^2$ by setting $\tilde{g}|\mathbb{R}^2=g$, and $\tilde{g}(\infty)=\infty$. $g$ and $\tilde{g}$ have exactly the same $k$-periodic points for any $k>1$. Let $\mathcal{O}$ be the $k$-periodic orbit contained in $C$. By Bonino's result there is a 2-periodic orbit $\mathcal{O}'\subseteq\mathbb{S}^2$ that is linked to $\mathcal{O}$. We will show that $\mathcal{O}'\cap C\neq\emptyset$. This will imply that $\mathcal{O}'\subseteq g^{-1}(C)\cup C$.

Refer to Figure 4. By contradiction suppose that $\mathcal{O}'\cap C=\emptyset$. Then also $\mathcal{O}'\cap \tilde{g}^{-1}(C)=\emptyset$. Consequently $\mathcal{O}'\subseteq U$ for the complementary domain $U$ of $C$. Since $C\cup\tilde{g}^{-1}(C)$ is acyclic, it is the intersection of a family of closed disks with diameters decreasing to $0$, and there is a Jordan curve $S\subseteq \mathbb{S}^2$ separating $\mathcal{O}'$ from $C\cup\tilde{g}^{-1}(C)$. Let $U_\infty$ and $U_b$ be the two complementary domains of $S$, with $\infty\in U_\infty$. Then $C\cup\tilde{g}^{-1}(C)$ is contained in $U_b$ and $\mathcal{O}'\subseteq U_\infty$. Note that $\tilde{g}(S)$ is a Jordan curve that bounds $C$ (therefore also $\mathcal{O}$), and $\mathcal{O}'$ is contained in $\tilde{g}(U_\infty)$, as $\tilde{g}(\mathcal{O}'\cup\{\infty\})=\mathcal{O}'\cup\{\infty\}$. 
\begin{figure}
	\centering
		\includegraphics[width=0.70\textwidth]{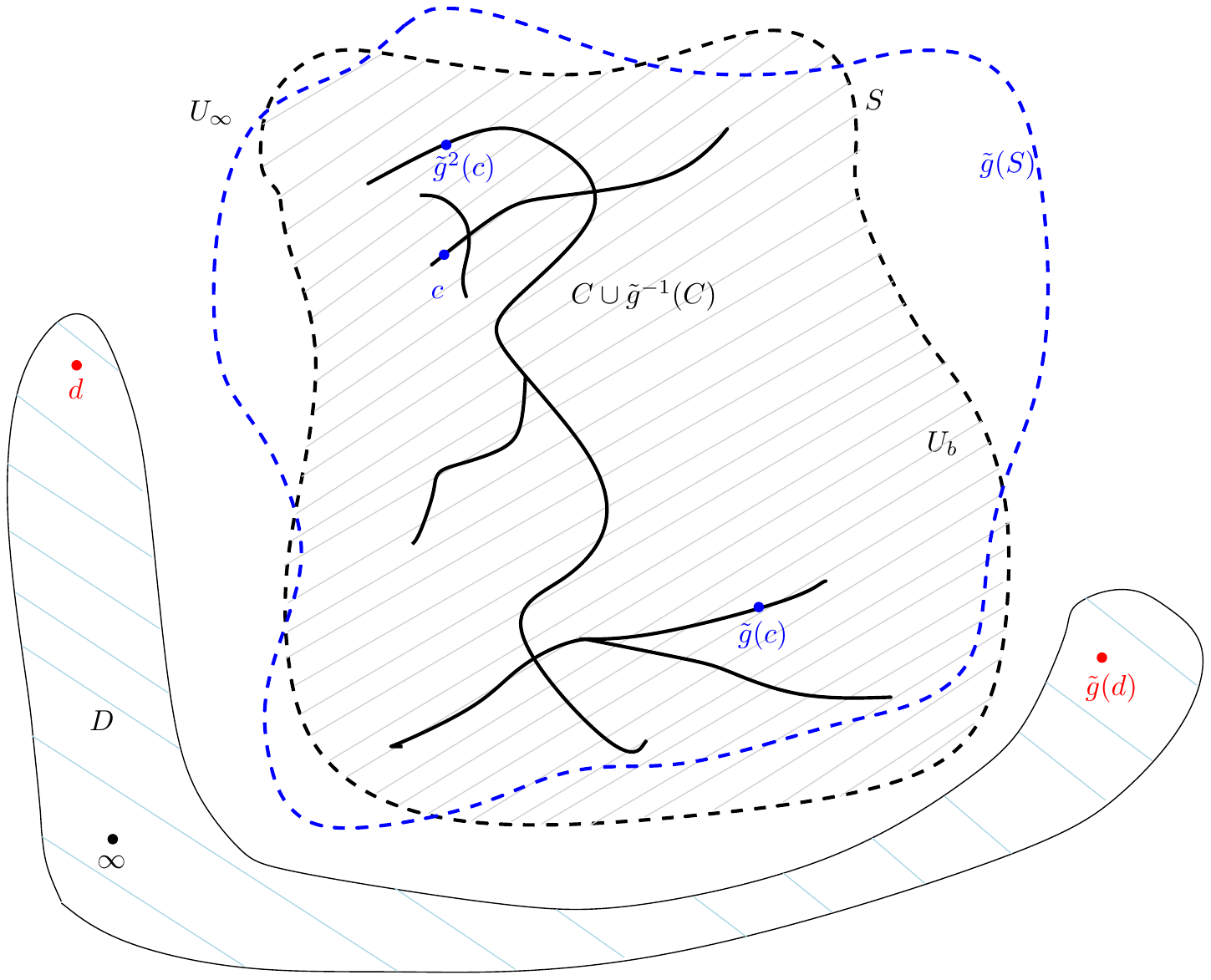}
	\label{fig:ostrovskiOR}
	\caption{Proof of Theorem \ref{main2} with $\mathcal{O}=\{c,\tilde{g}(c),\tilde{g}^2(c)\}$ and $\mathcal{O}'=\{d,\tilde{g}(d)\}$.}
\end{figure}
Let $D$ be a disk neighborhood of $\infty$ in $\mathbb{S}^2$ such that $\mathcal{O}'\subseteq D$ and $D$ is disjoint with $S\cup \tilde{g}(S)$. $D$ can be contracted to the point $\infty$. Then $\mathbb{S}^2\setminus (C\cup\{\infty\})$ is an open annulus that contains $S$ and $\tilde{g}(S)$ as essential Jordan curves. Therefore $S$ and $\tilde{g}(S)$ are freely isotopic in $\mathbb{S}^2\setminus (C\cup\{\infty\})$ and consequently in $\mathbb{S}^2\setminus (\mathcal{O}\cup\mathcal{O}')$. But this means that $\mathcal{O}$ and $\mathcal{O}'$ are not linked, leading to a contradiction.
\end{proof}
Finally we derive Corollary \ref{main3}.
\begin{proof}(of Corollary \ref{main3})
Let $C$ be a component of $h(X)\cap X$. First note that $h^{-1}(C)\cup C$ is an acyclic continuum by 1-dimensionality of $X$ and the fact that $h^{-1}(C)\cup C\subseteq X$. To prove (1) simply notice that by Theorem \ref{main1} $h$ will have a fixed point in each component of $X\cap h(X)$ that contains a periodic orbit. Since there are $n$ such components we deduce that there must be at least $n$ fixed points.

To prove (2) suppose that $C_1, C_2$ and $C_3$ are three components of $X\cap h(X)$ and $\mathcal{O}_1, \mathcal{O}_2$ and $\mathcal{O}_3$ are $k$-periodic orbits ($k>2$)  such that $\mathcal{O}_i\subseteq C_i$  for $i=1,2,3$. By Theorem \ref{main2} there are $2$-periodic orbits $\mathcal{O'}_1, \mathcal{O'}_2$ and $\mathcal{O'}_3$ such that $C_i\cap \mathcal{O'}_i\neq\emptyset$ for $i=1,2,3$. It suffices to show that $\mathcal{O'}_i\neq\mathcal{O'}_j$ for some $i,j\in\{1,2,3\}$. By contradiction suppose that $\mathcal{O'}_1=\mathcal{O'}_2=\mathcal{O'}_3$. Let $p\in \mathcal{O'}_1\cap C_1$. Then $p\notin C_2\cup C_3$. Therefore $h(p)\in C_2\cap C_3$. We obtain a contradiction since $C_2$ and $C_3$ are disjoint components of $X\cap h(X)$. 
\end{proof}

\section{Examples}\label{example}
 In this section we discuss restrictions on potential generalizations of our results. We shall exhibit that in Theorem \ref{main2} it is neither enough to require that $C$ contains an infinite (non-periodic) orbit, nor under the assumptions of Theorem \ref{main2} one can infer the existence of a fixed point.
\begin{theorem}
There is an orientation reversing homeomorphism $H$ and a continuum $C$ such that 
\begin{enumerate}
\item $H^{-1}(C)\cup C$ is an acyclic continuum and contains a $2k$-periodic orbit for every $k\geq 1$,
\item all orbits of $H$ are bounded,
\item $H$ has no fixed points.
\end{enumerate}
\end{theorem}
\begin{proof} Refer to Figure 5. Let $D=\{(x_1,x_2)|x_1^2+x_2^2\leq 1\}$. Consider a homeomorphism $\phi:D\to D$ that has periodic orbits of all periods. (One such example can be easily obtained by \cite{BM1} and \cite{BaM2} from the tent map $f$ on the unit interval. The inverse limit $Y_f=\lim_{\leftarrow}\{f,[0,1]\}$ is the buckethandle continuum and the shift on the inverse limit extends to a disk homeomorphism with $Y_f$ as a chaotic attractor and all periods present). One can also easily assure that $\phi(x)=x$ for every $x\in\partial D$. Let $D_1=\{(x_1,x_2)|(x_1-2)^2+x_2^2\leq 1\}$ and $D_{-1}=\{(x_1,x_2)|(x_1+2)^2+x_2^2\leq 1\}$. Let $g(x_1,x_2)=\phi(x_1-2,x_2)$ for $(x_1,x_2)\in D_1$, $g(x_1,x_2)=\phi(-x_1+2,x_2)$ for $(x_1,x_2)\in D_{-1}$ and $g(x)=x$ if $x\notin(D_{1}\cup D_{-1})$. Set $h=g\circ r$, where $r$ is the reflection about the $y$-axis. Set $E=\{(x_1,x_2)|x_1^2+4x_2^2\leq 1\}$. Consider the continuum $X=D_1\cup D_{-1}\cup E$. To obtain the desired homeomorphism $H$, one easily modifies $h$ in the infinite strip $S=\{(x_1,x_2):|x_1|<1\}$, without any changes in the complement of $S$, so that every point in $S$ is moved slightly upward, yet $H^{-1}(C)\cup C$ is acyclic (note that $H(C)\cap C$ contains the 2-periodic orbit $\{(-1,0),(1,0)\}$). In fact this last modification can be obtained using Boyles' homeomorphism from \cite{By2} (since in her example all points outside of $S$ are $2$-periodic), which will guarantee that all orbits are bounded, as promised. Clearly $H$ has no fixed points, as all points in $S$ are moved and $\mathbb{R}^2\setminus S$ consists of two $2$-periodic disjoint closed regions $U_{-1}$ and $U_1$; i.e. $H(U_{-1})=U_1$ and $H(U_1)=U_{-1}$.
\begin{figure}
	\centering
		\includegraphics[width=0.90\textwidth]{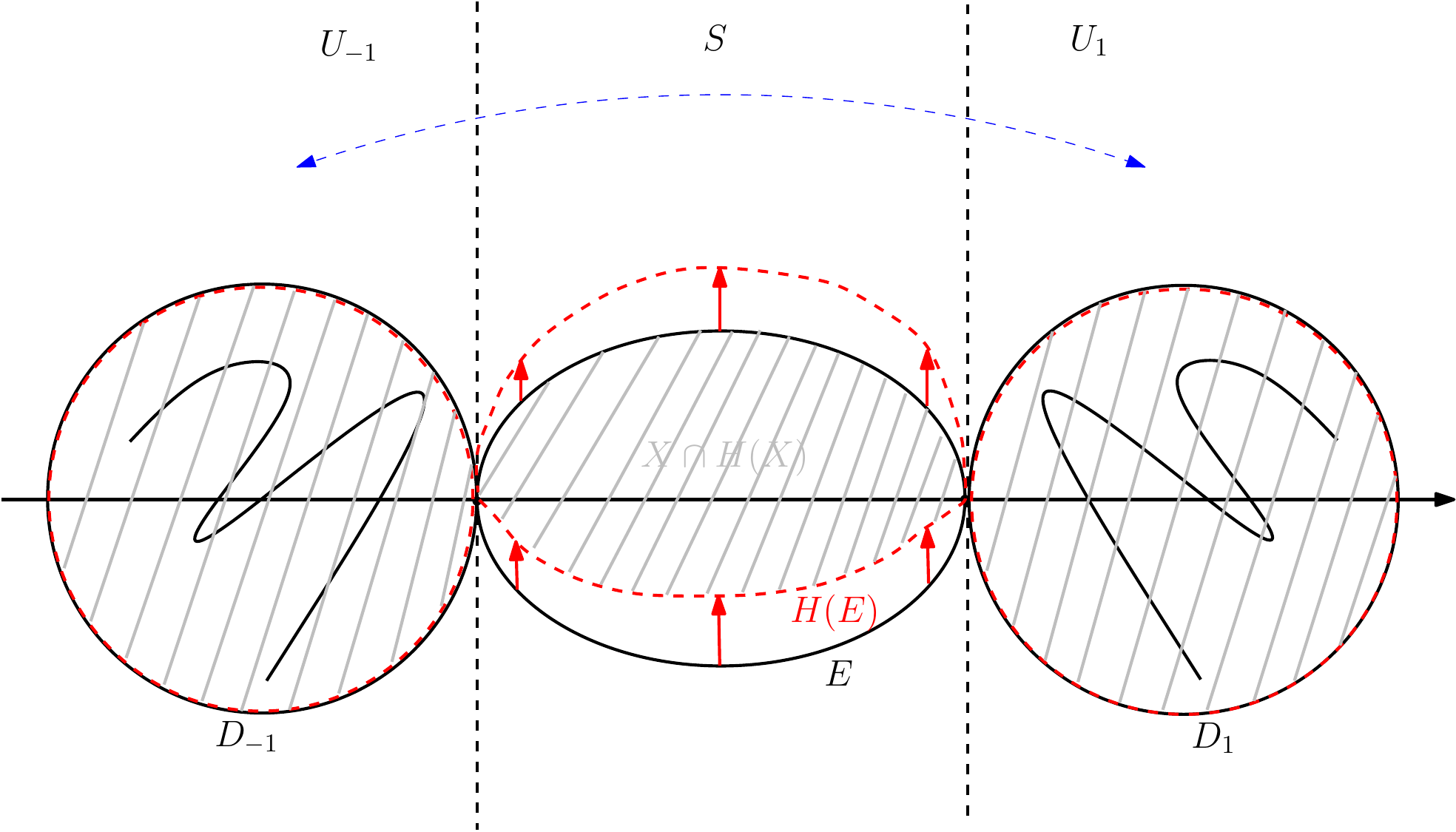}
	\label{fig:cloexample}
	\caption{The construction of the orientation reversing homeomorphism $H$.} 
\end{figure}
\end{proof}
 \begin{example}
There is an orientation reversing homeomorphism $h:\mathbb{R}^2\to\mathbb{R}^2$, and a continuum $C$ such that $h^{-1}(C)\cup C$ is an acyclic continuum and contains an infinite (non-periodic) orbit of $h$, but $h$ has no points of least period $2$. 
\end{example}
\begin{proof}
Consider the orientation reversing planar homeomorphism given by $h(x,y)=(-x-x|x|, y+y|1-y|)$. Then the arc $C=\{(x,y):x=0,y\in[0,1]\}$ is invariant under $h$, any $c\in C\setminus\{(0,0),(0,1)\}$ has an infinite (non-periodic) orbit in $C$ but there are no points of least period $2$ for $h$. 
\end{proof}
Note that the above example shows that in the assumptions of Bonino's result \cite{Bo2} one needs a $k$-periodic orbit ($k>1$), and non-periodic orbits do not force linked $2$-periodic orbits. We conclude with the following question.
\vspace{0.2cm}

\noindent
{\bfseries Question.} {\itshape Suppose $h : \mathbb{R}^2 \to  \mathbb{R}^2$ is an orientation preserving planar homeomorphism and $X$ is an acyclic continuum. Let $C$ be a component of $X\cap h(X)$. If there is a $c\in C$ such that $\{h^{-i}(c):i\in\mathbb{N}\}\subseteq C$ or $\{h^{-i}(c):i\in\mathbb{N}\}\subseteq C$ must $C$ also contain a fixed point of $h$?}
\section{Acknowledgments}
The author is indebted to the referee for careful reading of this paper, insightful comments and the many invaluable suggestions that greatly helped to improve the paper. 

This work was supported by the European Regional Development Fund in the IT4Innovations Centre of Excellence project (CZ.1.05/1.1.00/02.0070). Furthermore, the author gratefully acknowledges the partial support from the MSK DT1 Support of Science and Research in the Moravian-Silesian Region 2013\&2014 (RRC /05/2013) and  (RRC/07/2014). 
\bibliographystyle{amsplain}

\end{document}